\documentclass[11pt]{article}
\usepackage{amsmath,amsfonts,amsthm,amssymb, mathtools}
\usepackage{mathrsfs, graphicx,color,latexsym, tikz, calc,
}
\usepackage{tikz-cd}
\usepackage{graphicx}
\usepackage{epstopdf}
\usepackage{enumerate}
\usepackage{caption}
\usepackage{stmaryrd}
\usepackage{hyperref}
\usepackage{cite}

\usetikzlibrary{shadows}
\usetikzlibrary{patterns,arrows,decorations.pathreplacing}

\usepackage[margin=2.5cm]{geometry}
\usepackage{ulem}

\newtheorem{theorem}{Theorem}[section]
\newtheorem{lemma}[theorem]{Lemma}

\newtheorem{conjecture}[theorem]{Conjecture}
\newtheorem{corollary}[theorem]{Corollary}
\newtheorem{remark}[theorem]{Remark}

\newtheorem{claim}{Claim}[section]

\allowdisplaybreaks[4]

\title{\bf \Large }
\date{ }
\title{{\bf \Large 
Subspace variations of the weighted skew Bollob\'as theorem}\footnote{E-mail addresses: \url{wuyjmath@163.com} (Y. Wu), \url{ytli0921@hnu.edu.cn} (Y. Li), \url{lulugdmath@163.com} (L. Lu), \url{fenglh@163.com} (L. Feng). }
\author{
{\small  Yongjiang Wu$^1$,\ \ Yongtao Li$^2$, \ \ Lu Lu$^1$,\ \ Lihua Feng$^1$
}\\[2mm]
\small $^1$School of Mathematics and Statistics, HNP-LAMA, Central South University\\
 \small Changsha, Hunan, 410083, China\\ 
 \small $^2$Yau Mathematical Sciences Center, Tsinghua University, Beijing, 100084, China\\
}  }

\begin{document}
\maketitle
\begin{abstract}  
Let $V$ be a finite-dimensional real vector space. A collection $\mathcal{P} = \{(A_i,B_i)\}_{i=1}^m$ of pairs of subspaces of $V$ is called a  skew Bollob\'as system if $\dim(A_i\cap B_i)=0$ for each $i\in [m]$ and 
$\dim(A_i\cap B_j)>0$ for all $1\leq i<j \leq m$. 
Assume that $V = V^{(1)}\oplus \cdots \oplus V^{(r)}$ and $\mathcal{P}= \{(A_i,B_i)\}_{i=1}^m$ is a skew Bollob\'as system of subspaces of $V$ satisfying $
A_i = \bigoplus_{k=1}^r (A_i \cap V^{(k)})$ and $ B_i = \bigoplus_{k=1}^r (B_i \cap V^{(k)})$ for each $i\in [m]$. Denote $a_{i,k} = \dim(A_i \cap V^{(k)})$ and $b_{i,k} = \dim(B_i \cap V^{(k)})$. Suppose that $a_{1,k} \le \cdots \le a_{m,k}$ and $b_{1,k} \ge \cdots \ge b_{m,k}$ for each $k\in [r]$. Using the exterior algebraic method developed by Lov\'{a}sz and Scott--Wilmer, we prove that 
$$
\sum_{i=1}^{m} \frac{1}{\prod_{k=1}^{r} \binom{a_{i,k}+b_{i,k}}{a_{i,k}}} \le 1 .
$$
This generalizes the results of Alon (JCTA, 1985) and Scott--Wilmer (JLMS, 2021) to multipart weighted setting.
Secondly, we solve a conjecture of Hegedüs (AJC, 2015) concerning projective subspaces, showing that any skew Bollob\'as system of projective subspaces in an $n$-dimensional projective space contains at most $2^{n+1} - 2$ pairs.
Thirdly, we prove that
if $\mathcal{P}= \{(A_i,B_i)\}_{i=1}^m$ is a skew Bollob\'as system of subspaces of $V$ with $a_i=\dim (A_i)$ and $b_i=\dim (B_i)$, then
$$
\sum_{i=1}^m \frac{1}{(a_i+ b_i+1)\binom{a_i+b_i}{a_i}} \le 1.
$$
This gives an extension to the subspace setting of the results of Hegedüs--Frankl (EUJC, 2024) and Yue (DM, 2026).
Finally, we extend the above inequality to systems of $d$-tuples of subspaces, giving a unified bound that implies the corresponding results  for $d$-tuples of subsets. 
\end{abstract}



\section{Introduction}

For two integers $a\leq n$, we denote  $[a, n]=\{a, a+1, \ldots, n\}$ and  $[n]=[1,n]$. For $r\in [0,n]$, let $\binom{[n]}{r}=\{A\subseteq [n]: |A|=r\}$.
A collection $\mathcal{P}=\{(A_i, B_i)\}_{i=1}^m$ of pairs of subsets of $[n]$ is called a  \textit{Bollob\'as  system}  if it satisfies 
$A_i\cap B_i=\emptyset \text{ for every } i \in [m]$, 
 and 
$A_i\cap B_j\neq\emptyset \text{ for all } 1\leq i\neq j \leq m.$
In 1965, Bollob\'as \cite{B65} proved the following celebrated result,  known as the Two Families Theorem.

\begin{theorem}[Bollob\'as \cite{B65}]\label{B65} 
Let  $\mathcal{P}=\{(A_i, B_i)\}_{i=1}^m$  be a  Bollob\'as system of subsets of $[n]$. Suppose that $|A_i|=a$ and $|B_i|=b$ for every $ i\in [m]$. Then $m\le {a+b \choose a}$. 
\end{theorem}

The first standard approach to proving Theorem \ref{B65} utilizes the combinatorial technique: the inductive argument \cite{B65}, the counting argument \cite{Kat1974} and the probabilistic method \cite[page 11]{AS2016}. The original proof \cite{B65} reveals a weighted version for pairs of non-uniform families; see \cite[page 66]{Bol1986}.

\begin{theorem}[Weighted version]\label{weighted-B65} 
Let  $\mathcal{P}=\{(A_i, B_i)\}_{i=1}^m$  be a  Bollob\'as system of subsets of $[n]$. Suppose that $|A_i|=a_i$ and $|B_i|=b_i$ for every $ i\in [m]$. Then  
$$
\sum_{i=1}^m \frac{1}{
\binom{a_i + b_i}{a_i}} \leq 1 .
$$
\end{theorem}

The second approach to proving Theorem \ref{B65} was introduced by Lov\'asz \cite{L77}, who used the anti-symmetric tensor product from exterior algebra method. This elegant algebraic method is also valid under the weaker assumption, saying that $A_i\cap B_j\neq \emptyset$ for all $i<j$, instead of for all $i\neq j$; see \cite[Theorems 5.6 and 6.12]{BF1992} for detailed arguments.
In the set-theoretic context, a system $\mathcal{P} = \{(A_i,B_i)\}_{i = 1}^m$ satisfying this weak assumption together with $A_i \cap B_i = \emptyset$ for all  $i\in [m]$ is termed a \textit{skew Bollobás system}.
Moreover, this algebraic method also gives a natural extension to the system of subspaces of a finite dimensional vector space. 
This observation was also illustrated by Frankl \cite{Fra1982} using the symmetric tensor product. We refer to \cite[page 167]{FT18}. 

Let $V$ be a real vector space with dimension $n$. 
 A collection $\mathcal{P}=\{(A_i, B_i)\}_{i=1}^m$ of pairs of subspaces of $V$ is called a  \textit{skew Bollobás system} if  
$\dim(A_i\cap B_i)=0 \text{ for every } i \in [m]$; 
and
$\dim(A_i\cap B_j)>0 \text{ for all } 1\leq i< j \leq m.$ In fact, 
the algebraic approaches of Lov\'{a}sz \cite{L77} and Frankl \cite{Fra1982} directly imply the following subspace extension of Theorem \ref{B65}.

\begin{theorem}[Skew version, Lov\'asz \cite{L77}]\label{thm-L77} 
Let  $\mathcal{P}=\{(A_i, B_i)\}_{i=1}^m$  be a  skew Bollob\'as system of subspaces of $V\cong \mathbb{R}^n$. Suppose that $\dim (A_i)=a$ and $\dim (B_i)= b$ for every $ i\in [m]$. Then  
$$
m \leq \binom{a + b}{a}.
$$
\end{theorem}

Using a standard construction, mapping a set to a subspace, we see that the above space system version implies the set system  version in Theorem \ref{B65}. 
Indeed, suppose that  $\mathcal{P}=\{(A_i, B_i)\}_{i=1}^m$  is a  Bollob\'as system of subsets of $[n]$ with $|A_i|=a$ and $|B_i|=b$ for every $ i\in [m]$. Let $\{e_1,\ldots,e_n\}$ be the standard basis of $\mathbb{R}^n$.
Mapping each set $A_i$ to the subspace $A_i'=\text{span}\{e_p: p\in A_i\}\subseteq \mathbb{R}^n$ and each $B_i$ to the subspace $B_i'=\text{span}\{e_q: q\in B_i\}\subseteq \mathbb{R}^n$, we see that   
$\mathcal{P}'=\{(A_i', B_i')\}_{i=1}^m$ is a  skew Bollob\'as system of subspaces of $\mathbb{R}^n$ with
$\dim (A_i')=a$ and $\dim (B_i')=b$. Thus, the inequality of Theorem \ref{thm-L77} holds for $\mathcal{P}'$, and it coincides with the bound of Theorem \ref{B65} for $\mathcal{P}$.

Motivated by Theorem \ref{thm-L77}, it is natural to extend the weighted version in Theorem \ref{weighted-B65} to the skew version for pairs of vector subspaces. 
Unfortunately, Babai and Frankl \cite[Exercise 5.1.1]{BF1992} showed that the skew version of Theorem \ref{weighted-B65} does not hold in general. 
By developing the exterior algebraic method, Scott and Wilmer \cite{SW2021} achieved the first breakthrough by establishing a skew version of Theorem \ref{weighted-B65} in a more general vector space setting under a suitable condition.

\begin{theorem}[Weighted skew version, Scott and Wilmer \cite{SW2021}]\label{SW2021} 
Let  $\mathcal{P}=\{(A_i, B_i)\}_{i=1}^m$  be a skew Bollob\'as system of subspaces of $V\cong \mathbb{R}^n$. Suppose that $\dim (A_i)=a_i$ and $\dim (B_i)=b_i$ for every $ i\in [m]$,  and that the sequences satisfy $a_1\leq a_2\leq\cdots \leq a_m$ and $b_1\geq b_2\geq\cdots \geq b_m$. Then  
$$
\sum_{i=1}^m \frac{1}{
\binom{a_i + b_i}{a_i}} \leq 1 .
$$
\end{theorem}

For more related variations, we refer to \cite{A85,F84,GKMNPTX2020,YKXZG2022, HF24,Y261, Hol2021, KMN2021} and references therein. At the last paragraph in \cite[Section 6]{SW2021}, Scott and Wilmer remarked that {\it``Can any of these variations be further addressed with exterior algebra methods?''} In this paper, we provide positive evidence for supporting Scott--Wilmer's question by establishing several new weighted skew Bollob\'as-type inequalities for systems of vector spaces. 
Using the exterior algebra method, we first show a unified extension of the results of Alon \cite{A85} and Scott--Wilmer \cite{SW2021}; see the forthcoming Section \ref{sec-Alon}.  Moreover, we shall prove some other related results for weighted skew version using different methods.

\subsection{Weighted inequalities for multipart structures} \label{sec-Alon}
The classical Two Families Theorem and its variants typically treat the ground set or vector space as a single entity. A significant step toward a more structured setting was taken by Alon \cite{A85}, who considered the case where the ground set is partitioned into several disjoint blocks.

\begin{theorem}[Alon \cite{A85}]\label{A85} 
Let $[n]$ be the disjoint union of some sets $X_1,\ldots, X_r$.
Let  $\mathcal{P}=\{(A_i, B_i)\}_{i=1}^m$  be a skew Bollob\'as system of subsets of $[n]$. Suppose that $|A_i\cap X_k|=a_k$ and $|B_i\cap X_k|=b_k$ for every $ i\in [m]$ and $k\in [r]$. Then 
$$m\leq \prod_{k=1}^r {a_k+b_k \choose a_k}.$$ 
\end{theorem}

Alon’s result  requires the intersection sizes $a_k, b_k$ to be constant across all pairs. A natural question is whether this uniformity condition can be relaxed to allow varying sizes while still obtaining a meaningful bound.
We answer this question by proving a weighted extension in the more general setting of vector spaces. The first main result of this paper  generalizes  Theorems \ref{SW2021} and \ref{A85}. 

\begin{theorem}\label{Main3} 
Let $V = V^{(1)} \oplus V^{(2)} \oplus \cdots \oplus V^{(r)}$ be a finite-dimensional real vector space, and let  $\mathcal{P}=\{(A_i, B_i)\}_{i=1}^m$ be a skew Bollob\'as system of subspaces of $V$ satisfying $
A_i = \bigoplus_{k=1}^r (A_i \cap V^{(k)})$ and $ B_i = \bigoplus_{k=1}^r (B_i \cap V^{(k)})$ for each $i\in [m]$. Denote
$a_{i,k} = \dim (A_i \cap V^{(k)})$ and $b_{i,k} = \dim (B_i \cap V^{(k)})$.
Assume that $
a_{1,k} \le a_{2,k} \le \cdots \le a_{m,k}$ and $b_{1,k} \ge b_{2,k} \ge \cdots \ge b_{m,k}$ for each $k\in [r]$. 
Then 
$$
\sum_{i=1}^m \frac{1}{\prod_{k=1}^r \binom{a_{i,k}+b_{i,k}}{a_{i,k}}} \le 1.
$$
\end{theorem}

Theorem \ref{Main3} recovers Theorem \ref{SW2021} by setting $r=1$. Moreover, Theorem \ref{Main3} implies  the following uniform case, which generalizes both Theorem \ref{thm-L77} and Theorem \ref{A85}.

\begin{corollary}\label{cr2} 
Let $V = V^{(1)} \oplus V^{(2)} \oplus \cdots \oplus V^{(r)}$ be a finite-dimensional real vector space, and let  $\mathcal{P}=\{(A_i, B_i)\}_{i=1}^m$ be a skew Bollob\'as system of subspaces of $V$ satisfying $
A_i = \bigoplus_{k=1}^r (A_i \cap V^{(k)})$ and $ B_i = \bigoplus_{k=1}^r (B_i \cap V^{(k)})$ for each $i\in [m]$. Suppose that
$ \dim (A_i \cap V^{(k)})=a_k$ and $ \dim (B_i \cap V^{(k)})=b_k$ for every $i\in [m]$ and $k\in [r]$.
Then 
$$m\leq \prod_{k=1}^r {a_k+b_k \choose a_k}.$$
\end{corollary}

Theorem \ref{Main3} yields  the following result for set systems via a standard construction.

\begin{corollary}\label{cr3} 
Let $[n]$ be the disjoint union of some sets $X_1,\ldots, X_r$.
Let  $\mathcal{P}=\{(A_i, B_i)\}_{i=1}^m$  be a skew Bollob\'as system of subsets of $[n]$. Suppose that $|A_i\cap X_k|=a_{i,k}$ and $|B_i\cap X_k|=b_{i,k}$ for every $ i\in [m]$ and $k\in [r]$, and that for each $k\in [r]$,
$
a_{1,k} \le a_{2,k} \le \cdots \le a_{m,k}$ and $b_{1,k} \ge b_{2,k} \ge \cdots \ge b_{m,k}.
$
Then 
$$
\sum_{i=1}^m \frac{1}{\prod_{k=1}^r \binom{a_{i,k}+b_{i,k}}{a_{i,k}}} \le 1.
$$
\end{corollary}

\subsection{Projective subspaces and Hegedüs' conjecture}
Let $\mathbb{F}$ be an arbitrary field and let $V$ be an $(n+1)$‑dimensional vector space over $\mathbb{F}$ with $n\geq 0$.
The \textit{projective space} of dimension $n$ is defined as  
$
\mathbb{P}(V)=\bigl(V \setminus \{0\}\bigr) / \sim 
$, where the equivalence relation   
$u \sim v$ if and only if $\exists\ \lambda\in\mathbb{F}\setminus \{0\}\ \text{such that}\ u = \lambda v.
$
An element $[v] \in \mathbb{P}(V)$ is called a \textit{projective point}, which corresponds to the one‑dimensional subspace  $\mathrm{span}\{v\}$.
A subset $A \subseteq \mathbb{P}(V)$ is called a \textit{projective subspace} if there exists a  subspace $U \subseteq V$, $U \neq \{0\}$, such that  
$
A = \mathbb{P}(U)=\{ [u] : u \in U \setminus \{0\} \}.
$ 
In 2015, Hegedüs \cite{H15} proposed the following conjecture for projective subspace.

\begin{conjecture}[Hegedüs \cite{H15}]\label{cj1} Let $W$ be an $n$-dimensional projective space over an arbitrary field $\mathbb{F}$ and let $\mathcal{P}=\{(A_i, B_i)\}_{i=1}^m$  be a skew Bollob\'as system of projective subspaces of $W$, i.e., $A_i\cap B_i=\emptyset \text{ for every } i \in [m]$; 
 and 
$A_i\cap B_j\neq\emptyset \text{ for all } 1\leq i< j \leq m$. Then  
$$
m\le 2^{\,n+1}-2 .
$$ 
\end{conjecture}
The term  skew Bollob\'as system is the same as the classical set-pair theorem. Here the objects are projective subspaces and the intersection conditions are imposed in the projective sense.  
The second main result of this paper is an affirmative solution of Hegedüs’s conjecture.

\begin{theorem}\label{Main4} Let $W$ be an $n$-dimensional projective space over an arbitrary field $\mathbb{F}$ and let $\mathcal{P}=\{(A_i, B_i)\}_{i=1}^m$  be a skew Bollob\'as system of projective subspaces of $W$. Then  
$$
m\le 2^{\,n+1}-2 .
$$ 
\end{theorem}

\subsection{Weighted inequalities for subspace pairs}

The weighted version in Theorem \ref{weighted-B65} does not hold under the weaker skew intersection condition, since Babai and Frankl \cite[Exercise 5.1.1]{BF1992} provided an example by considering all pairs $(A,[n]\setminus A)$, where $A$ runs over all subsets of $[n]$, arranged in decreasing order of their size. Then $\mathcal{P}=\{(A_i,[n]\setminus A_i)\}_{i=1}^{2^n}$ forms a skew Bollob\'as system of subsets of $[n]$, and $\sum_{i=1}^{2^n} 1/ {a_i +b_i \choose a_i}=n+1$.

Recently, Hegedüs and Frankl \cite{HF24}  provided an interesting variant of Theorem \ref{weighted-B65} by showing that if $\mathcal{P}=\{(A_i, B_i)\}_{i=1}^m$ is a skew Bollob\'as system of subsets of $[n]$, then 
\begin{equation} \label{eq-HF} \sum_{i=1}^m \frac{1}{\binom{a_i + b_i}{a_i}} \leq n + 1 .
\end{equation}
Using a probabilistic method, Yue \cite{Y261} removed the explicit dependence on $n$. 

\begin{theorem}[Yue \cite{Y261}]\label{HFY} 
Let  $\mathcal{P}=\{(A_i, B_i)\}_{i=1}^m$  be a skew Bollob\'as system of subsets of $[n]$. Suppose that $|A_i|=a_i$ and $|B_i|=b_i$ for every $ i\in [m]$. Then    
$$
\sum_{i=1}^m \frac{1}{\bigl(a_i +b_i+1\bigr)
\binom{a_i + b_i}{a_i}} \leq 1 .
$$
\end{theorem}

As the third result of this paper, we extend Theorem \ref{HFY} to the following setting for subspaces, which substantially removes the monotonicity condition of Theorem \ref{SW2021}.

\begin{theorem}\label{tm1} 
Let  $\mathcal{P}=\{(A_i, B_i)\}_{i=1}^m$  be a skew Bollob\'as system of subspaces of $V\cong \mathbb{R}^n$. Suppose that $\dim (A_i)=a_i$ and $\dim (B_i)=b_i$ for every $ i\in [m]$. Then  
$$
\sum_{i=1}^m \frac{1}{\bigl(a_i +b_i+1\bigr)
\binom{a_i + b_i}{a_i}} \leq 1 .
$$
\end{theorem}

By the standard construction of mapping sets to  subspaces, Theorem \ref{tm1} implies Theorem \ref{HFY}. 
Our proof in the subspace version is motivated by the set version of Hegedüs and Frankl \cite{HF24}, where we reduce the present problem to the previous Lov\'{a}sz theorem by a clever argument on maximality assumption.  
This is different from the probabilistic proof of Theorem \ref{HFY}.

In 1984, Füredi \cite{F84} employed so-called general position arguments  to generalize Theorem \ref{thm-L77} to the $t$-intersecting setting. 
Let $\mathcal{P}=\{(A_i, B_i)\}_{i=1}^m$  be a collection of pairs of subspaces of  $V$. For a non-negative integer $t$, we say that  
$\mathcal{P}$  is  a \textit{skew Bollob\'as $t$-system}  if 
 $\dim(A_i \cap B_i) \leq t \text{ for any } i \in [m]$ and
$\dim(A_i \cap B_j) > t \text{ for any } 1\leq i< j \leq m$.
The case $t=0$ reduces to the ordinary skew Bollob\'as  system. We refer to \cite{SW2021,YKXZG2022} for related results. 
Using Füredi's general position arguments \cite{F84}, we extend Theorem \ref{tm1}  to this $t$-intersecting setting.

\begin{corollary} \label{cr1}
Let  $\mathcal{P}=\{(A_i, B_i)\}_{i=1}^m$  be a skew Bollob\'as $t$-system of subspaces of $V\cong \mathbb{R}^n$. Suppose that $\dim (A_i)=a_i$ and $\dim (B_i)=b_i$ for every $ i\in [m]$, and that $a_m, b_1\geq t$. Then  
$$
\sum_{i=1}^m \frac{1}{\bigl(a_i +b_i-2t+1\bigr)
\binom{a_i + b_i-2t}{a_i-t}} \leq 1 .
$$
\end{corollary}

Setting $t = 0$ in Corollary \ref{cr1} reduces to Theorem \ref{tm1}. 
Moreover, the condition $a_m,b_1 \geq t$ is natural in this setting, as it ensures that the binomial coefficients are well-defined.

\subsection{Weighted inequalities for $d$-tuples of subspaces}

The Bollob\'as two families theorem admits a natural extension. Let $d\geq 2$ and $\mathcal{P} = \{(A_i^{(1)},\dots,A_i^{(d)})\}_{i=1}^m$ be a collection of $d$-tuples of subsets of  $[n]$. We say that $\mathcal{P}$  is a \textit{Bollob\'as  system}  if 
\begin{itemize}
\item[(i)]
$ A_i^{(p)} \cap  A_i^{(q)}=\emptyset \text{ for any } i \in [m] \text{ and } 1\leq p<q\leq d$; and 

\item[(ii)]
for any $1\leq i\neq j \leq m$, there exist $1\leq p< q \leq d$ such that
$A_i^{(p)}\cap A_j^{(q)}\neq \emptyset$. 
\end{itemize}
Moreover, $\mathcal{P}$ is called a \textit{skew Bollob\'as system} 
if the above condition $i\neq j$ is replaced by $i<j$. 
For $d=2$, these definitions reduce to the classical  set-pair systems. For non-negative integers $a_1, \ldots, a_d$ and $n\geq a_1+\cdots+a_d$, the multinomial coefficient is denoted by 
$
\binom{n}{a_1,\ldots, a_d} := \frac{n!}{a_1!\cdots a_d!(n-a_1-\cdots-a_d)!}$. 
Extending (\ref{eq-HF}), 
Hegedüs and Frankl \cite{HF24} established the tight bound
$$
\sum_{i=1}^{m}
\frac{1}{\displaystyle\binom{a_i^{(1)}+\cdots+a_i^{(d)}}{a_i^{(1)},\dots ,a_i^{(d)}}
      }
\leq \binom{n+d-1}{d-1}.
$$
The following theorem extends Theorem \ref{HFY} to a skew Bollob\'as  system of $d$-tuples of subsets. 

\begin{theorem}[Yue \cite{Y262}]\label{HFY2} 
Let $\mathcal{P} = \{(A_i^{(1)},\dots,A_i^{(d)})\}_{i=1}^m$ be a skew Bollob\'as  system of $d$-tuples of subsets of $[n]$.  Suppose that $|A_i^{(k)}|=a_i^{(k)}$ for every $ i\in [m]$ and $k\in [d]$.
Then
$$
\sum_{i=1}^{m}
\frac{1}{\displaystyle\binom{a_i^{(1)}+\cdots+a_i^{(d)}}{a_i^{(1)},\dots ,a_i^{(d)}}
      \binom{a_i^{(1)}+\cdots+a_i^{(d)}+d-1}{d-1}}
\leq 1.
$$
\end{theorem}

The notion of  a Bollob\'as system can be naturally extended from sets to vector spaces. Let $V$ be a real vector space of dimension $n$, and let $\mathcal{P} = \{(A_i^{(1)},\dots,A_i^{(d)})\}_{i=1}^m$ be a collection of $d$-tuples of subspaces of  $V$. Then
$\mathcal{P}$  is called  a \textit{Bollob\'as system}  if 
 \begin{itemize}

\item[(i)] $\dim(A_i^{(1)}+\cdots +A_i^{(d)})=\dim(A_i^{(1)})+\cdots +\dim(A_i^{(d)}) \text{ for any } i \in [m],$ and

\item[(ii)] for any $1\leq i\neq j \leq m$, there exist $1\leq p< q \leq d$ such that
$\dim (A_i^{(p)}\cap A_j^{(q)})>0.$
\end{itemize}

If the condition (ii) is required only for ordered pairs $i<j$, then $\mathcal{P}$ is called a \textit{skew Bollob\'as system}. We remark that the above condition (i) is slightly different from that of the previous $d$-tuples of subsets. 
This notation was originally introduced by Oum and Wee \cite{OW2018}. 
As the fourth result of this paper, we show the following vector space extension of Theorem \ref{HFY2}.

\begin{theorem}\label{tm2}
Let $\mathcal{P} = \{(A_i^{(1)},\dots,A_i^{(d)})\}_{i=1}^m$ be a skew Bollob\'as  system of $d$-tuples of subspaces of  $V\cong \mathbb{R}^n$.  Suppose that $\dim(A_i^{(k)})=a_i^{(k)}$ for every $ i\in [m]$ and $k\in [d]$.
Then
$$
\sum_{i=1}^{m}
\frac{1}{\displaystyle\binom{a_i^{(1)}+\cdots+a_i^{(d)}}{a_i^{(1)},\dots ,a_i^{(d)}}
      \binom{a_i^{(1)}+\cdots+a_i^{(d)}+d-1}{d-1}}
\leq 1.
$$
\end{theorem}

Theorem \ref{HFY2} follows from Theorem \ref{tm2} via the  standard construction.

\paragraph{Our approaches.} 
For Theorems \ref{Main3} and \ref{Main4}, 
we employ the exterior algebra approach of Lov\'asz \cite{L77}, Scott and Wilmer \cite{SW2021}. The multipart structure in Theorem \ref{Main3}  demands a refined analysis, as  projection dimensions in different components must be controlled simultaneously. We overcome this by a component‑wise induction that carefully interleaves projections and wedge products,  tracking dimensions separately in each block and combining them via a product formula.
Theorem \ref{Main4} is proved by associating each projective subspace with a corresponding non-trivial vector subspace and then applying a linear independence argument in the exterior algebra over an arbitrary field.  

The proofs of Theorems \ref{tm1} and \ref{tm2}  develop  the approach introduced by Hegedüs and Frankl \cite{HF24},  modifing it in two essential aspects. Firstly, we transfer the problem to the vector space setting, overcoming the difficulty that the collection of all subspaces is infinite by working with a suitably constructed finite family. 
 Secondly, we work with a more general weighted sum and design an  extension operation that  does not increase this sum. We introduce a carefully chosen bounded potential function that strictly increases with each operation. Since the potential is bounded, the process must terminate after finitely many steps. The final configuration then satisfies a known bound, yielding the required inequality.
Our method is distinct from the probabilistic  technique used by Yue \cite{Y261, Y262} for set systems, and  applies directly to vector space configuration.

The remainder of the paper is organized as follows. Section \ref{Sea} collects the necessary preliminary results from exterior algebra.   Section \ref{se3}  presents the proofs of Theorems \ref{Main3} and \ref{Main4}.  Section \ref{se2} contains the proofs of Theorems \ref{tm1} and \ref{tm2}, together with Corollary \ref{cr1}.

\section{Preliminary results}\label{Sea}
In this section, we review some basic definitions and properties of exterior algebras that will be used throughout the proof; see, e.g., \cite{SW2021}. Let $V$ be an $n$-dimensional real vector space and let $$\bigwedge V=\bigoplus_{r=0}^n\bigwedge^r V$$
denote the standard grading of its   exterior algebra.
For a fixed ordered basis $F = (f_{1},\dots,f_{n})$ of  $V$ and a subset $A\in \binom{[n]}{r}$, we denote
$f_{A} := \bigwedge_{a\in A} f_a\in \bigwedge^r V$, 
 where the elements of $A$ are listed in increasing order. 
The collection $\{f_{A} : A \in \binom{[n]}{r}\} $  forms a basis of $\bigwedge^{r} V$.  In particular,
$
\dim (\bigwedge^{r} V) = \binom{n}{r}
$
and $\dim (\bigwedge V) =2^n$.  
Moreover, a crucial property states that for $A, B\subseteq [n]$, 
\begin{align} 
f_A\wedge f_B= 
\begin{cases}
0 & \text{ if } A\cap B \neq \emptyset, \\
(-1)^{\rho(A, B)}f_{A\cup B}   & \text{ otherwise},
\end{cases}\label{new1}
\end{align}
where $\rho(A, B)=|\{(a,b): a\in A, b\in B, a>b\}|$. 
For a subspace $A\subseteq V$ with basis $\{a_1,\ldots ,a_r\}$, we define $v_A =a_1\wedge \cdots \wedge a_r \in \bigwedge^r V$. We remark that $v_A$ is only determined up to a non-zero constant when we choose a different basis of $A$. This will not cause any ambiguity in our arguments. For two subspaces $B, C\subseteq V$, equation (\ref{new1}) implies that
$v_B\wedge v_C=0$ if and only if $\dim(B\cap C)>0$.

For two subspaces $U,W\subseteq \bigwedge V$, we define $U\wedge W=\mathrm{span}\{u\wedge w: u\in U, w\in W\}$. 
By establishing a correspondence between hypergraphs and subspaces, 
 Scott and Wilmer \cite{SW2021} proved the following analogue of the local LYM inequality in the exterior algebra.

\begin{lemma}\cite{SW2021}\label{SW1}
 Let $V\cong \mathbb{R}^{n}$ and let $W$ be a subspace of $\bigwedge^{r}V$. Then for $c\in [0, n - r]$,
$$
\frac{\dim(W\wedge\bigwedge^{c}V)}{\binom{n}{r+c}}\geq\frac{\dim (W)}{\binom{n}{r}}.
$$
\end{lemma}

Denote by  
$
\mathcal{B}(V)= \bigl\{(f_1,\dots,f_n) : \{f_1,\dots,f_n\} \text{ is a basis of } V \bigr\}
$ 
the set of  ordered bases of $V$.
When $V=\mathbb{R}^n$, we naturally identify  $\mathcal{B}(V)$ with $\mathrm{GL}_n(\mathbb{R})$.
Given $F=(f_1,\dots,f_n)\in \mathcal{B}(V)$
and $J \subseteq [n]$, we define $V_J = \operatorname{span}\{ f_j : j \in J \}$ and  the \textit{linear projection} $\pi_J^F : V \to V_J$ by
$$
\pi_J^F\left( \sum_{j=1}^n \alpha_j f_j \right) = \sum_{j \in J} \alpha_j f_j.
$$ 
It is easy to see that $\dim \pi_{J}^F (C)\le \min \{\dim C,|J|\}$ for any subspace $C$ of $V$. 
Motivated by the rank argument of F\"{u}redi \cite{F84} (see \cite[page 171]{FT18}), Scott and Wilmer \cite{SW2021} proved the following result. 

\begin{lemma}\cite{SW2021}\label{SW2}
Let $V = \mathbb{R}^{n}$ and let $C_1,\ldots ,C_m$ be proper subspaces of $V$. Then there exists a nonzero polynomial $G$ in the $n^2$ variables $f_{ij}(1\leq i,j\leq n)$ such that $G(F)\neq 0$ implies $F =$ $(f_{ij})\in \mathrm{GL}_n(\mathbb{R})$ and
$
\dim \pi_{J}^{F}(C_{i}) = \min \{\dim C_{i},|J|\}
$ 
for all $i\in [m]$ and $J\subseteq [n]$.
\end{lemma}

We extend the  projection $\pi_J^F$ linearly from $V$ to $\bigwedge^r V$ as follows: for any $r \geq 0$ and $A \in \binom{[n]}{r}$, we define 
$
\pi_J^F(f_A) =\bigwedge_{a \in A} \pi_J^F(f_a),
$
and extend linearly to all of $\bigwedge V$. 
Explicitly, it follows that 
$$
\pi_J^F(f_A) = 
\begin{cases}
f_A & \text{ if } A \subseteq J, \\
0   & \text{ otherwise}.
\end{cases}
$$

 Scott and Wilmer \cite{SW2021} proved the following properties. 
 
\begin{lemma}\cite{SW2021}\label{SW3}
 Let $V = \mathbb{R}^{n}$ and let $W \subseteq \bigwedge^{r} V$  be a linear subspace. Then there exists a nonzero polynomial $H$ in the $n^2$ variables $f_{ij}, 1 \leq i, j \leq n$, such that $H(F) \ne 0$ implies that $F = (f_{ij}) \in \mathrm{GL}_n(\mathbb{R})$, and for every $ m \in [n-1]$ and  all $J \in \binom{[n]}{m}$,
$$
\dim \bigl(\pi_{J}^{F}(W) \bigr)=\max_{J^*\in \binom{[n]}{m},F^*\in \mathrm{GL}_n(\mathbb{R})} \dim \bigl(\pi_{J^*}^{F^*}(W)\bigr).
$$
\end{lemma}

\begin{lemma}\cite{SW2021}\label{ad3}
Suppose that $0< r\le n-d$. Let $V \cong \mathbb{R}^{n}$ and let  $W\subseteq \bigwedge^r V$ be a linear subspace and $F\in \mathcal{B}(V)$. Then 
$$
 \max_{J\in \binom{[n]}{n-d}}\frac{\dim \bigl(\pi_{J}^{F}(W)\bigr)}{ \binom{n-d}{r}}\geq \frac{\dim (W)}{ \binom{n}{r}}.
$$
\end{lemma}

\begin{lemma}
    \cite{SW2021} \label{ad3-2}
    Suppose that $0<r \le n-d$. Let $V \cong \mathbb{R}^{n}$ and let  $W\subseteq \bigwedge^r V$ be a linear subspace and $F\in \mathcal{B}(V)$. Then 
    \[  \frac{\dim(W\wedge\bigwedge^{d}V)}{ \binom{n}{r+d}}\geq
 \frac{1}{\binom{n}{n-d}}\sum_{J\in \binom{[n]}{n-d}}\frac{\dim \bigl(\pi_{J}^{F}(W)\bigr)}{ \binom{n-d}{r}}. \]
\end{lemma}

To prove Theorem \ref{Main3}, we need to introduce 
the following lemma concerning the dimension of wedge products over direct sums. It will be crucial for handling the multipart structure.
\begin{lemma}\label{ad1}
 Let $V \cong \mathbb{R}^{n}$ and $V = V^{(1)} \oplus \cdots \oplus V^{(r)}$. 
For each $k\in [r]$, let $d_k\geq 0$ and let $U_k \subseteq \bigwedge^{d_k} V^{(k)}$ be a subspace. 
Then
$$
\dim\left( \bigwedge_{k=1}^r U_k \right) = \prod_{k=1}^r \dim (U_k).
$$
\end{lemma}

\begin{proof}
The case for $d_k= 0$ or $d_k>\dim (V^{(k)})$ is trivial.
From now on, we assume that $\dim (V^{(k)})=n_k$ and $1\leq d_k\leq n_k$.
Let $m_k = \dim (U_k)$ and fix a basis 
$\{u_{1,k}, u_{2,k}, \dots, u_{m_k,k}\}$ for each $U_k$.
The collection
$$
G = \bigl\{ u_{j_1,1} \wedge u_{j_2,2} \wedge \cdots \wedge u_{j_r,r} : 1 \le j_k \le m_k \bigr\}
$$
spans $\bigwedge_{k=1}^r U_k$ and contains $\prod_{k=1}^r m_k$ vectors. 
It suffices to prove that $G$ is linearly independent.

For each $k$, choose a basis $\{e_1^{(k)}, \dots, e_{n_k}^{(k)}\}$ of $V^{(k)}$.
For a subset $A \subseteq [n_k]$ with $|A| = d_k$, write
$
e_A^{(k)} = \bigwedge_{t \in A} e_t^{(k)}, 
$
where the elements of $A$ are listed in increasing order.
The collection $\bigl\{e_{A_k}^{(k)} : A_k\in \binom{[n_k]}{d_k}\bigr\}$ is a basis of  $\bigwedge^{d_k} V^{(k)}$.
Since $V = \bigoplus_{k=1}^r V^{(k)}$, the set
$
\bigl\{ e_{A_1}^{(1)} \wedge e_{A_2}^{(2)} \wedge \cdots \wedge e_{A_r}^{(r)}: A_k\in \binom{[n_k]}{d_k} \bigr\}
$
 is linearly independent in $\bigwedge^{\,d_1 + \dots + d_r} V$.
For each $u_{j_k,k} \in U_k$, write
$$
u_{j_k,k} = \sum_{ A_k\in \binom{[n_k]}{d_k} } \lambda_{j_k,k}^{A_k} \, e_{A_k}^{(k)}, \qquad \lambda_{j_k,k}^{A} \in \mathbb{R}.
$$
For each fixed \(k\), the vectors \(u_{1,k},\dots,u_{m_k,k}\) are linearly independent. 
Hence, the $m_k \times \binom{n_k}{d_k}$ matrix $\bigl( \lambda_{j_k,k}^{A_k} \bigr)_{j_k,A_k}$ has full row rank $m_k$. 
Consequently, we can choose $m_k$ distinct column indices 
$A_k^{(1)},\dots,A_k^{(m_k)}$ such that the square matrix
$
M_k := \left( \lambda_{j_k,k}^{A_k^{(t_k)}} \right)_{j_k,t_k\in [m_k]}
$
is invertible.

Assume there exist scalars $c_{j_1,\dots,j_r}$ such that
\begin{align}
\sum_{j_1=1}^{m_1} \cdots \sum_{j_r=1}^{m_r} c_{j_1,\dots,j_r} \; 
u_{j_1,1} \wedge \cdots \wedge u_{j_r,r} = 0.\label{fo1} 
\end{align}
Expanding each wedge product gives
$$
u_{j_1,1} \wedge \cdots \wedge u_{j_r,r}
= \sum_{A_1\in \binom{[n_1]}{d_1}} \cdots \sum_{A_r\in \binom{[n_r]}{d_r}} 
\Bigl( \prod_{k=1}^r \lambda_{j_k,k}^{A_k} \Bigr) \,
e_{A_1}^{(1)} \wedge \cdots \wedge e_{A_r}^{(r)} .
$$
Substituting into (\ref{fo1}) and using the linear independence of the vectors $e_{A_1}^{(1)} \wedge \cdots \wedge e_{A_r}^{(r)}$ yields
\begin{align}
\sum_{j_1\in [m_1],\dots,j_r\in[m_r]} c_{j_1,\dots,j_r} \;
\prod_{k=1}^r \lambda_{j_k,k}^{A_k} = 0 \qquad \text{for all } A_1\in \binom{[n_1]}{d_1},\dots,A_r\in \binom{[n_r]}{d_r}.\label{fo2} 
\end{align}
For each $t_k \in [m_k]$, setting $A_k = A_k^{(t_k)}$ in (\ref{fo2}) gives
$$
\sum_{j_1\in [m_1],\dots,j_r\in[m_r]} c_{j_1,\dots,j_r} \;
\prod_{k=1}^r \lambda_{j_k,k}^{A_k^{(t_k)}} = 0 \qquad \text{for all } t_1 \in [m_1],\dots,t_r\in [m_r]. 
$$
The coefficient matrix of this system  is precisely the Kronecker product
$
M:= M_1^{\top} \otimes M_2^{\top} \otimes \cdots \otimes M_r^{\top}.
$
where $
M_k = \left( \lambda_{j_k,k}^{A_k^{(t_k)}} \right)_{j_k,t_k\in [m_k]}.$
Since each $M_k$ is invertible,  $M$ is invertible. 
Thus, $
c_{j_1,\dots,j_r} = 0 \text{ for all } j_1 \in [m_1],\dots,j_r\in [m_r].
$
Consequently, the set $G$ is linearly independent.
\end{proof}

\section{Proofs of  Theorems \ref{Main3} and \ref{Main4}}\label{se3} 

\subsection{Proof of Theorem \ref{Main3}}\label{sse1}

\begin{proof}[\bf Proof of Theorem \ref{Main3}]
Let $V = V^{(1)} \oplus V^{(2)} \oplus \cdots \oplus V^{(r)}$ with $\dim(V^{(k)})=n_k$.
For every pair $(A_i, B_i)$,  let $A_{i,k}=A_i\cap V^{(k)}$ and $B_{i,k}=B_i\cap V^{(k)}$.
Then
$a_{i,k} = \dim (A_{i,k})$ and $b_{i,k} = \dim (B_{i,k})$.
  For each integer $k\in [r]$, we define recursively subspaces $W_i^{(k)}\subseteq\bigwedge^{a_{i,k}}V^{(k)}$ by setting $W_0^{(k)}=\{0\}$ and 
\begin{equation} \label{eq-defi} W_{i+1}^{(k)}= \operatorname{span}\left\{W_i^{(k)}\wedge\bigwedge^{a_{i+1,k}-a_{i,k}}V^{(k)},\ v_{A_{i+1,k}}\right\}.
\end{equation}

Fix $k\in [r]$, we work entirely within $V^{(k)}\cong \mathbb{R}^{n_k}$.
Choose a linear isomorphism $\phi_k: V^{(k)} \rightarrow \mathbb{R}^{n_k}$.
Extend this isomorphism naturally to the exterior algebra 
$\phi_k: \bigwedge V^{(k)} \to \bigwedge  \mathbb{R}^{n_k}.$
Firstly, apply Lemma \ref{SW2} to the finite family
$
\mathcal{C}_k=\bigl\{\phi_k(A_{i,k}),\;\phi_k(B_{i,k}),\; \phi_k(A_{i,k}+B_{i,k}): 
i\in [m]\bigr\}.
$
This yields a nonzero polynomial $G_k$  in the entries of an $ n_k\times n_k$  matrix $ \hat{F}_k$.  
Secondly,   apply Lemma \ref{SW3} to the subspace $\phi_k(W_i^{(k)} )\subseteq \bigwedge^{a_{i,k}} \mathbb{R}^{n_k}$.  
This gives a nonzero polynomial $H_{i,k}(\hat{F_k})$. Define    
$$
P_k(\hat{F}_k) =  G_k(\hat{F}_k) \prod_{i=1}^{m}  H_{i,k}(\hat{F}_k).
$$
Since each factor is a nonzero polynomial, $P_k$ is nonzero. 
Hence, there exists a basis 
$\hat{F}_k= (\hat{f}_1^{(k)},\dots,\hat{f}_{n_k}^{(k)})$ of $\mathbb{R}^{n_k}$ such that $P_k(\hat{F}_k) \neq 0$.
Now pull back this basis to $V^{(k)}$ via $\phi_k^{-1}$. Define
$$
F_k = (f_1^{(k)},\dots,f_{n_k}^{(k)}), \quad \text{where } f_j^{(k)} = \phi_k^{-1}(\hat{f}_j^{(k)}) \in V^{(k)}.
$$
Since $\phi_k$ is an isomorphism, the basis $F_k$ inherits the desired properties from $\hat{F}_k$. 
In particular,  for every  $J\subseteq [n_k]$ and all $i\in[m]$, by Lemmas   \ref{SW2} and \ref{SW3}, we have
\begin{align}
\dim (\pi_{J}^{F_k}(A_{i,k}))& = \min \{a_{i,k},|J|\},\ \dim (\pi_{J}^{F_k}(B_{i,k}))= \min \{b_{i,k},|J|\},\label{foa7}\\
\dim (\pi_{J}^{F_k}(A_{i,k}+B_{i,k}))& = \min \{a_{i,k}+b_{i,k},|J|\},\label{foa4}\\
\dim \Bigl(\pi_{J}^{F_k}(W_i^{(k)})\Bigr)&=\max_{J^*\in \binom{[n_k]}{|J|},F^*\in \mathcal{B}(V^{(k)})} \dim \Bigl(\pi_{J^*}^{F^*}(W_i^{(k)} )\Bigr).\label{foa6}
\end{align}
Choose such an $F_k$ for every $V^{(k)}$.
Then we obtain a  basis of $V$:
$$
F = (f_1^{(1)},\dots,f_{n_1}^{(1)},f_1^{(2)},\dots,f_{n_2}^{(2)},\dots,f_1^{(r)},\dots,f_{n_r}^{(r)}).
$$

For every $i\in [m]$ and  $k\in [r]$, let $n_{i,k}=a_{i,k}+b_{i,k}$ and $ a_i = \sum_{k=1}^r a_{i,k}$. We denote $V_{n_{i,k}}= \pi_{[n_{i,k}]}^{F_k}(V^{(k)})\subseteq V^{(k)}$. Moreover, we define the projected spaces:  
\begin{align*}
 Z_{i}^{(k)}&=\pi_{[n_{i,k}]}^{F_k}(W_i^{(k)})\subseteq \bigwedge^{a_{i,k}} V_{n_{i,k}},\\
X_{i}^{(k)}&=\pi_{[n_{i+1,k}]}^{F_k}(W_i^{(k)})\subseteq \bigwedge^{a_{i,k}} V_{n_{i+1,k}},\\ Y_{i}^{(k)}&=X_{i}^{(k)}\wedge\bigwedge^{a_{i+1,k}-a_{i,k}}V_{n_{i+1,k}} \subseteq \bigwedge^{a_{i+1,k}} V_{n_{i+1,k}}.
\end{align*}
A key observation in our proof is to construct the spaces $Z_i$ and $Y_i$ as follows: 
\begin{align*}
Z_i =\bigwedge_{k=1}^r Z_i^{(k)} \quad \text{and}\quad Y_i= \bigwedge_{k=1}^r Y_i^{(k)}.
\end{align*}
By Lemma \ref{ad1}, we have
\begin{align}
\dim(Z_i)=\prod_{k=1}^r \dim ( Z_i^{(k)}) \quad \text{and} \quad \dim(Y_i)=\prod_{k=1}^r \dim ( Y_i^{(k)}).\label{foa}
\end{align}

\begin{claim}\label{cla1}
For any $i\in [0,m-1]$,  we have $\dim(Z_{i+1})\geq \dim(Y_{i})+1$.
\end{claim}
\begin{proof}[Proof of Claim \ref{cla1}]
First, we define projections on $V$. For each $i\in [m]$, we define
$
\pi_i: V \to V_i := \bigoplus_{k=1}^r V_{n_{i,k}}
$ 
as the direct sum of projections by setting that for each  $v = \sum_{k=1}^r v^{(k)} \in V$, 
$$
\pi_i(v) := \sum_{k=1}^r \pi_{[n_{i,k}]}^{F_k}(v^{(k)})
$$
In the sequel, we need to extend this projection naturally to the exterior algebra $\bigwedge V$ by setting $\pi_i: \bigwedge V \to \bigwedge  V_i$, e.g., for $w=v\wedge u$, then $\pi_i (w)=\pi_i(v)\wedge \pi_i (u)$. 

Since $A_{i+1} = \bigoplus_{k=1}^r A_{i+1,k}$, we may write
$v_{A_{i+1}}=\bigwedge_{k=1}^{r}v_{A_{i+1,k}}\in\bigwedge^{a_{i+1}}V$.
Consider its projection $\pi_{i+1}(v_{A_{i+1}}) = \bigwedge_{k=1}^r \pi_{[n_{i+1,k}]}^{F_k}(v_{A_{i+1,k}})$.
By definition, we have 
\begin{align*}
Z_{i+1}^{(k)} = &\operatorname{span}\left\{ \pi_{[n_{i+1,k}]}^{F_k}\Bigl(W_i^{(k)}\wedge\bigwedge^{a_{i+1,k}-a_{i,k}}V^{(k)}\Bigr),\ \pi_{[n_{i+1,k}]}^{F_k}(v_{A_{i+1,k}}) \right\}\\
=&\operatorname{span}\left\{ Y_i^{(k)},\ \pi_{[n_{i+1,k}]}^{F_k}(v_{A_{i+1,k}}) \right\}.
\end{align*}
For every $k\in [r]$, we denote $d_k = \dim(Z_{i+1}^{(k)})$ and $y_k = \dim(Y_i^{(k)})$.
Then
\begin{align}
d_k = 
\begin{cases}
y_k & \text{if } \pi_{[n_{i+1,k}]}^{F_k}(v_{A_{i+1,k}}) \in Y_i^{(k)}, \\
y_k + 1 & \text{otherwise}.\label{foa2}
\end{cases}
\end{align}
Let $S = \{ k \in [r] : \pi_{[n_{i+1,k}]}^{F_k}(v_{A_{i+1,k}}) \notin Y_i^{(k)} \}$. 
We claim that $S\neq \emptyset$. 
Otherwise, if $ S = \emptyset$, then
\begin{align}
\pi_{i+1}(v_{A_{i+1}}) = \bigwedge_{k=1}^r \pi_{[n_{i+1,k}]}^{F_k}(v_{A_{i+1,k}}) \in \bigwedge_{k=1}^r Y_i^{(k)} = Y_i.\label{foa3}
\end{align}
We denote 
$W_i =\bigwedge_{k=1}^r W_i^{(k)}$.
Then $Y_{i}^{(k)}=\pi_{[n_{i+1,k}]}^{F_k}\Bigl(W_i^{(k)}\wedge\bigwedge^{a_{i+1,k}-a_{i,k}}V^{(k)}\Bigr)\subseteq \pi_{[n_{i+1,k}]}^{F_k}(W_i^{(k)})\wedge\bigwedge^{a_{i+1,k}-a_{i,k}}\pi_{i+1}(V)$ implies
$Y_i\subseteq \pi_{i+1}(W_i)\wedge\bigwedge^{a_{i+1}-a_{i}}\pi_{i+1}(V).$
By the definition (\ref{eq-defi}), we get 
\[  W_{i}^{(k)}=\mathrm{span}\left\{ v_{A_{1,k}}\wedge \bigwedge^{a_{i,k}-a_{1,k}}V^{(k)}, \ \ v_{A_{2,k}}\wedge \bigwedge^{a_{i,k}-a_{2,k}}V^{(k)},\  
\ldots \ , \ v_{A_{i,k}}\right\}, \]
which implies that any element of  $W_i\wedge\bigwedge^{a_{i+1}-a_{i}} V$ is a linear combination of elements $\bigl\{v_{A_{h}}\wedge\bigwedge^{a_{i+1}-a_{h}}V: h\leq i\bigr\}$. Combining equation (\ref{foa3}) and $Y_i\subseteq \pi_{i+1}(W_i)\wedge\bigwedge^{a_{i+1}-a_{i}}\pi_{i+1}(V)$, we obtain that  $\pi_{i+1}(v_{A_{i+1}})$ is a linear combination of elements $\bigl\{\pi_{i+1}(v_{A_{h}})\wedge\bigwedge^{a_{i+1}-a_{h}}\pi_{i+1}(V): h\leq i\bigr\}$.
The skew Bollob\'{a}s assumption says that $\dim(A_h\cap B_{i+1})>0$ for all $h\leq i$. It follows  that $v_{A_{h}}\wedge v_{B_{i+1}}=0$ and   $\pi_{i+1}(v_{A_{h}})\wedge \pi_{i+1}(v_{B_{i+1}})=0$ for all $h\leq i$. 
This leads to $\pi_{i+1}(v_{A_{i+1}})\wedge \pi_{i+1}(v_{B_{i+1}})=0$.

On the other hand, 
by equations (\ref{foa7}) and (\ref{foa4}), we know that for every $k\in [r]$, 
$$ \pi_{[n_{i+1,k}]}^{F_k}(A_{i+1,k})\cap\pi_{[n_{i+1,k}]}^{F_k}(B_{i+1,k}) = \{0\}.$$
By the property (\ref{new1}) and the definition of $\pi_{i+1}$, it follows that $\pi_{i+1}(v_{A_{i+1}})\wedge \pi_{i+1}(v_{B_{i+1}})\neq 0$, which contradicts with the previous argument.
Therefore, we must have $ S\neq \emptyset$.

If there exists some $k\in [r]$ such that $y_k=0$, then $\dim(Z_{i+1})\geq 1=\dim(Y_{i})+1$ holds  because $d_k\geq 1$ for all $k\in [r]$.
Now assume that $y_k\geq 1$ for all $k\in [r]$. Since $ S\neq \emptyset$, applying equations (\ref{foa}) and (\ref{foa2}) to compute the dimension difference yields
\[ 
\dim(Z_{i+1}) - \dim(Y_i) = \prod_{k=1}^r d_k - \prod_{k=1}^r y_k
= \left( \prod_{k \notin S} y_k \right) \left( \prod_{k \in S} (y_k + 1) - \prod_{k \in S} y_k \right) \geq 1. \qedhere 
\] 
\end{proof}

\begin{claim}\label{cla2}
For any $i\in [m-1]$ and $k\in [r]$, by denoting $n_{i,k}=a_{i,k}+b_{i,k}$, we have  
$$
\frac{\dim(Y_i^{(k)})}{\binom{n_{i+1,k}}{a_{i+1,k}}}
\geq \frac{\dim(Z_i^{(k)})}{\binom{n_{i,k}}{a_{i,k}}}.$$ 
\end{claim}

\begin{proof}[Proof of Claim \ref{cla2}]
To start with, we recall the assumption that 
$a_{1,k}\le a_{2,k}\le \cdots \le a_{m,k}$ and $b_{1,k}\ge b_{2,k} \ge \cdots \ge b_{m,k}$ for every $k\in [r]$. Next, 
we divide the proof into three cases according to the relation between $n_{i,k}$ and $n_{i+1,k}$, where $n_{i,k}=a_{i,k} +b_{i,k}$. 

{\bf Case 1.} If $n_{i,k}= n_{i+1,k}$, i.e., $b_{i,k}- b_{i+1,k}= a_{i+1,k} - a_{i,k}\geq 0$,  then $X_i^{(k)}=Z_i^{(k)}\subseteq \bigwedge^{a_{i,k}} V_{n_{i,k}}$ and $Y_i^{(k)}= Z_i^{(k)} \wedge \bigwedge^{a_{i+1,k}-a_{i,k}} V_{n_{i,k}}$.
Since $V_{n_{i,k}}\cong \mathbb{R}^{n_{i,k}}$,  Lemma \ref{SW1} gives
$$
\frac{\dim(Y_i^{(k)})}{\binom{n_{i+1,k}}{a_{i+1,k}}}
= \frac{\dim\left( Z_i^{(k)} \wedge \bigwedge^{a_{i+1,k}-a_{i,k}} V_{n_{i,k}} \right)}{\binom{n_{i,k}}{a_{i+1,k}}}
\geq \frac{\dim(Z_i^{(k)})}{\binom{n_{i,k}}{a_{i,k}}}.
$$

{\bf Case 2.}
If $n_{i,k}> n_{i+1,k}$, i.e., $b_{i,k}- b_{i+1,k}>a_{i+1,k} - a_{i,k}\geq 0$, then for any $J\in \binom{[n_{i,k}]}{n_{i+1,k}}$, we have 
$\pi_{J}^{F_k}(W_i^{(k)})=\pi_{J}^{F_k}\Bigl(\pi_{[n_{i,k}]}^{F_k}(W_i^{(k)})\Bigr)=\pi_{J}^{F_k}(Z_i^{(k)})$. 
By equation (\ref{foa6}), for every $J\in \binom{[n_{i,k}]}{n_{i+1,k}}$,  
\begin{align*}
\dim \Bigl(\pi_{J}^{F_k}(W_i^{(k)} )\Bigr)= \dim \Bigl(\pi_{[n_{i+1,k}]}^{F_k}(W_i^{(k)} )\Bigr)=\dim \Bigl(\pi_{[n_{i+1,k}]}^{F_k}(Z_i^{(k)})\Bigr).
\end{align*}
 Let $F_k^i = (f_1^{(k)},\dots,f_{n_{i,k}}^{(k)})$ be the restriction of the basis $F_k$ to the first $n_{i,k}$ vectors. Consequently,
\begin{align}
\dim \Bigl(\pi_{[n_{i+1,k}]}^{F_k}(Z_i^{(k)})\Bigr)&= \max_{J\in \binom{[n_{i,k}]}{n_{i+1,k}}} \dim \Bigl(\pi_{J}^{F_k}(W_i^{(k)} )\Bigr)=\max_{J\in \binom{[n_{i,k}]}{n_{i+1,k}}} \dim \Bigl(\pi_{J}^{F_k}(Z_i^{(k)} )\Bigr)\notag\\
&=\max_{J\in \binom{[n_{i,k}]}{n_{i+1,k}}} \dim \Bigl(\pi_{J}^{F_k^i}(Z_i^{(k)} )\Bigr).\label{fi2}
\end{align}
Observe  that $Z_{i}^{(k)}\subseteq \bigwedge^{a_{i,k}} V_{n_{i,k}}$ and $F_k^i\in \mathcal{B}(V_{n_{i,k}})$.
Now we apply Lemmas \ref{SW1} and \ref{ad3} to obtain
\begin{align*}
\frac{\dim(Y_i^{(k)})}{\binom{n_{i+1,k}}{a_{i+1,k}}}
&= \frac{\dim\Bigl( X_i^{(k)} \wedge \bigwedge^{a_{i+1,k}-a_{i,k}} V_{n_{i+1,k}} \Bigr)}{\binom{n_{i+1,k}}{a_{i,k}+(a_{i+1,k}-a_{i,k})}}
\overset{\text{Lemma \ref{SW1}}}{\geq}\ \frac{\dim(X_i^{(k)})}{\binom{n_{i+1,k}}{a_{i,k}}}\\
&=\frac{\dim\Bigl(\pi_{[n_{i+1,k}]}^{F_k}(Z_i^{(k)})\Bigr)}{\binom{n_{i+1,k}}{a_{i,k}}}\overset{\text{(\ref{fi2})}}{=}\max_{J\in \binom{[n_{i,k}]}{n_{i+1,k}}}\frac{\dim \Bigl(\pi_{J}^{F_k^i}(Z_i^{(k)})\Bigr)}{\binom{n_{i,k}-(n_{i,k}-n_{i+1,k})}{a_{i,k}}}\overset{\text{Lemma \ref{ad3}}}{\geq}\ \frac{\dim(Z_i^{(k)})}{\binom{n_{i,k}}{a_{i,k}}}.
\end{align*}

{\bf Case 3.}
If $n_{i,k}<n_{i+1,k}$, i.e., $a_{i+1,k} - a_{i,k}>b_{i,k}- b_{i+1,k}\geq 0$, then for any $J\in \binom{[n_{i+1,k}]}{n_{i,k}}$, we have 
$\pi_{J}^{F_k}(W_i^{(k)})=\pi_{J}^{F_k}\Bigl(\pi_{[n_{i+1,k}]}^{F_k}(W_i^{(k)})\Bigr)=\pi_{J}^{F_k}(X_i^{(k)})$. 
Let $c=b_{i,k}- b_{i+1,k}\geq 0$ and $d=a_{i+1,k} - a_{i,k}-c\geq 0$. Then $n_{i+1,k}=n_{i,k}+d$. 
Equation (\ref{foa6}) implies that for every $J\in \binom{[n_{i+1,k}]}{n_{i,k}}$,  
\begin{align}
\dim \Bigl(\pi_{J}^{F_k}(W_i^{(k)} )\Bigr)= \dim \Bigl(\pi_{[n_{i,k}]}^{F_k}(W_i^{(k)} )\Bigr)=\dim(Z_i^{(k)}).\label{fi1}
\end{align}
Let $F_k^{i+1} = (f_1^{(k)},\dots,f_{n_{i+1,k}}^{(k)})\in \mathcal{B}(V_{n_{i+1,k}})$.
Note that $X_{i}^{(k)}\subseteq \bigwedge^{a_{i,k}} V_{n_{i+1,k}}$.
Applying Lemmas \ref{SW1} and \ref{ad3-2}, we obtain
\begin{align*}
\frac{\dim(Y_i^{(k)})}{\binom{n_{i+1,k}}{a_{i+1,k}}}
&= \frac{\dim\Bigl( X_i^{(k)} \wedge \bigwedge^{c+d} V_{n_{i+1,k}} \Bigr)}{\binom{n_{i+1,k}}{a_{i,k}+c+d}}
\overset{\text{Lemma \ref{SW1}}}{\geq}\ \frac{\dim\Bigl( X_i^{(k)} \wedge \bigwedge^{d} V_{n_{i+1,k}} \Bigr)}{\binom{n_{i+1,k}}{a_{i,k}+d}}\\
&\overset{\text{Lemma \ref{ad3-2}}}{\geq}\ \frac{1}{\binom{n_{i+1,k}}{n_{i+1,k}-d}}\sum_{J\in \binom{[n_{i+1,k}]}{n_{i+1,k}-d}}\frac{\dim \Bigl(\pi_{J}^{F_k^{i+1}}(X_i^{(k)})\Bigr)}{ \binom{n_{i+1,k}-d}{a_{i,k}}}=\frac{1}{\binom{n_{i+1,k}}{n_{i,k}}}\sum_{J\in \binom{[n_{i+1,k}]}{n_{i,k}}}\frac{\dim \Bigl(\pi_{J}^{F_k}(X_i^{(k)})\Bigr)}{ \binom{n_{i,k}}{a_{i,k}}}\\
&=\frac{1}{\binom{n_{i+1,k}}{n_{i,k}}}\sum_{J\in \binom{[n_{i+1,k}]}{n_{i,k}}}\frac{\dim \Bigl(\pi_{J}^{F_k}(W_i^{(k)})\Bigr)}{ \binom{n_{i,k}}{a_{i,k}}}\overset{\text{(\ref{fi1})}}{=}\frac{\dim(Z_i^{(k)})}{\binom{n_{i,k}}{a_{i,k}}}.\qedhere
\end{align*}
\end{proof}

We now return to show the proof of Theorem \ref{Main3}. 
By Claim \ref{cla2} and equation (\ref{foa}), we obtain that for any $i\in [m-1]$,
$$
\frac{\dim(Y_{i})}{\prod_{k=1}^r \binom{n_{i+1,k}}{a_{i+1,k}}}\geq \frac{\dim(Z_{i})}{\prod_{k=1}^r \binom{n_{i,k}}{a_{i,k}}}.
$$
This together with  Claim \ref{cla1} implies
\begin{align*}
1&\geq \frac{\dim(Z_m)}{\prod_{k=1}^r \binom{n_{m,k}}{a_{m,k}}}\geq \frac{1}{\prod_{k=1}^r \binom{n_{m,k}}{a_{m,k}}}+\frac{\dim(Y_{m-1})}{\prod_{k=1}^r \binom{n_{m,k}}{a_{m,k}}}\geq \frac{1}{\prod_{k=1}^r \binom{n_{m,k}}{a_{m,k}}}+\frac{\dim(Z_{m-1})}{\prod_{k=1}^r \binom{n_{m-1,k}}{a_{m-1,k}}}\geq \cdots\\
&\geq \sum_{i=1}^m \frac{1}{\prod_{k=1}^r \binom{n_{i,k}}{a_{i,k}}}=\sum_{i=1}^m \frac{1}{\prod_{k=1}^r \binom{a_{i,k}+b_{i,k}}{a_{i,k}}}.
\end{align*}
Thus, the desired weighted bound is satisfied, completing the proof of Theorem \ref{Main3}. 
\end{proof}

\subsection{Proof of Theorem \ref{Main4}}
The definition of the exterior algebra given in 
subsection \ref{sse1} extends directly to an arbitrary field  $\mathbb{F}$, only the underlying scalars change.
Let $V$ be an $n$-dimensional vector space over $\mathbb{F}$.  
Its exterior algebra is  
$
\bigwedge V = \bigoplus_{r=0}^{n} \bigwedge^r V.
$ 
For any ordered basis $F=(f_1,\dots,f_{n})$ of $V$ and a subset $A\in\binom{[n]}{r}$, we write  
$
f_A = \bigwedge_{a\in A} f_a \in \bigwedge^r V,
$
 where the wedge product is taken with indices in increasing order.  
The collection $\{f_A : A\in\binom{[n]}{r}\}$ is a basis of $\bigwedge^r V$. Consequently,  
$
\dim_{\mathbb{F}}\Bigl(\bigwedge^r V\Bigr) = \binom{n}{r}.
$ 
Let $C$ be a $d$-dimensional subspace of $V$ with a basis $\{c_1,\ldots,c_d\}$ and let 
$
v_C=c_1\wedge\cdots \wedge c_d\in \bigwedge^{d}V.
$ 
Note that $v_C$ depends on the choice of the basis and any two such expressions differ only by a nonzero scalar factor. 
The following basic property of the wedge products of vectors of subspaces will be essential.

\begin{lemma}[{see \cite[Lemma 6.11 and Exercise 6.16]{BF1992} or \cite{HF24}}]\label{BT1}
Let  $\mathbb{F}$  be an arbitrary field and let $V$ be an $n$-dimensional vector space over $\mathbb{F}$. Let $B$ and $C$ be subspaces of $V$. Then
$v_B\wedge v_C=0$ if and only if $B\cap C\neq \{0\}$.
\end{lemma}

We shall also need a simple criterion for linear independence, often called the triangular criterion.

\begin{lemma}[{see \cite[Proposition 2.9]{BF1992} or  \cite{HF24}}]\label{TC}
Let  $\mathbb{F}$  be an arbitrary field. Let $V$ and $T$ be linear spaces over $\mathbb{F}$ and $\Omega$ an arbitrary set. Let  $g : V \times \Omega \to T$ be a function that is linear in the first variable. For $i\in [m]$, let $v_i \in V$ and $a_i \in \Omega$ be such that
$$
g(v_i, a_j) 
\begin{cases} 
\neq 0 & \text{ if } i = j,\\
=0 & \text{ if } i < j.
\end{cases}
$$
Then $v_1, \ldots, v_m$ are linearly independent.
\end{lemma}

\begin{proof}[\bf Proof of Theorem \ref{Main4}]
Let $\mathcal{P}=\{(A_i, B_i)\}_{i=1}^m$  be a skew Bollob\'as system of projective subspaces of $W$.
Let $V$ be an $(n+1)$-dimensional vector space over $\mathbb{F}$  such that $W = \mathbb{P}(V)$.  
For any $i\in [m]$, each projective subspace $A_i$ (resp. $B_i$) corresponds uniquely to a vector subspace $U_i\subseteq V$ (resp. $V_i\subseteq V$) with $U_i\neq\{0\}, V_i\neq\{0\}$.  
The skew Bollob\'as conditions translate to  
$$
U_i\cap V_i = \{0\},\quad U_i\cap V_j \neq \{0\}\ (i<j).
$$
By Lemma \ref{BT1}, we have  
$$
v_{U_i}\wedge v_{V_i} \neq 0,\quad 
v_{U_i}\wedge v_{V_j}  = 0\ (i<j).
$$
It follows from Lemma \ref{TC} that $v_{U_1},\dots,v_{U_m}$ are linearly independent in $\bigwedge V$.
Since $U_i\neq\{0\}, V_i\neq\{0\}$ and $U_i\cap V_i = \{0\}$, we have $U_i\neq V$ and $ V_i\neq V$.
Combining $U_i\neq\{0\}$ and $U_i\neq V$, we obtain $1\le\dim_{\mathbb{F}} (U_i)\leq n$.  
Hence, each $v_{U_i}$ lies in the subspace  
$
S := \bigoplus_{k=1}^{n} \bigwedge^{k}V \subsetneq \bigwedge V .
$ 
From $\dim_{\mathbb{F}} (V)=n+1$, we know that
$$
\dim_{\mathbb{F}}(S) = \sum_{k=1}^{n}\binom{n+1}{k}
                    = 2^{\,n+1} - 2 .
$$
Since $v_{U_1},\dots,v_{U_m}$ are linearly independent   elements of $S$, we obtain  
$
m \le \dim_{\mathbb{F}}(S)= 2^{\,n+1} - 2.
$
\end{proof}

\section{Proofs of  Theorems \ref{tm1} and \ref{tm2}}\label{se2}
\subsection{Proofs of Theorem \ref{tm1} and Corollary \ref{cr1}}
To set up the proof, we first recall a simple lemma.

\begin{lemma}\cite{HF24}\label{HFt1} 
Let  $\mathcal{P}=\{(A_i, B_i)\}_{i=1}^m$  be a skew Bollob\'as system of pairs  of subspaces of $V\cong \mathbb{R}^n$.
Then
$$
m\leq 2^{n}.$$
\end{lemma}

Now we are ready to prove Theorem \ref{tm1}.
Our proof of Theorem \ref{tm1} has its roots in \cite{HF24}. This is different from the probabilistic method of Theorem \ref{HFY}.

\begin{proof}[\bf Proof of Theorem \ref{tm1}]
The argument is divided into three steps.

\textbf{Step 1. Existence of a  maximal-weight system.} 
 Let $D=\{(a,b): a, b\in [0,n] \text{ and }  a+b\le n\}$.  
For a skew Bollob\'as system $\mathcal{P}=\{(A_i, B_i)\}_{i=1}^m$ with $\dim (A_i)=a_i$ and $\dim (B_i)=b_i$, we define its type
$
\tau(\mathcal{P})=\bigl(m;(a_1,b_1),(a_2,b_2),\dots,(a_m,b_m)\bigr)
$
 and its weight
$$
\omega(\mathcal{P})=\sum_{i=1}^{m}\frac{1}{(1+a_i+b_i)\binom{a_i+b_i}{a_i}} .
$$
The list $(a_1,b_1),\dots,(a_m,b_m)$ is ordered and repetitions are allowed.  By Lemma \ref{HFt1},  we have $m\le 2^n$. 
Because $m\le 2^n$ and each $(a_i,b_i)\in D$, the number of different types is at most $\sum_{m=0}^{2^n}|D|^m$, denote this finite set by $T$. 
For a type $\tau=\bigl(m;(a_1,b_1),\dots,(a_m,b_m)\bigr)$, define  
$$
\omega(\tau)=\sum_{i=1}^{m}\frac{1}{(1+a_i+b_i)\binom{a_i+b_i}{a_i}} .
$$
If $\tau$ is  realizable (i.e., some system $\mathcal{P}$ has type $\tau$), then every system of that type has the same weight $\omega(\tau)$.  
The set of realizable types is a subset of the finite set $T$, hence $\{\omega(\tau):\tau \text{ is realizable}\}$ is a finite set of real numbers. Consequently, its maximum is attained. Choose a realizable type $\tau^*$ with $\omega(\tau^*)=\max \omega(\tau)$.
 We may  assume that $\mathcal{P}=\{(A_i, B_i)\}_{i=1}^m$ has type $\tau^*$. Then  $\omega(\mathcal{P})=\max \omega(\tau)$.

 \textbf{Step 2.  An extension operation that preserves weight.}  
Suppose that $A_i+B_i\neq V$ for some index $i$.  Then there exists a vector $x\in V\backslash (A_i+B_i)$. Since $(A_i,B_i)\in  \mathcal{P}$ and $\mathcal{P}$  is a skew Bollob\'as system, we infer that $
\bigl(A_i\oplus\langle x\rangle, B_i\bigr), \bigl(A_i, B_i\oplus\langle x\rangle\bigr) \notin \mathcal{P}$. 
Replacing the pair $(A_i, B_i)$ in $\mathcal{P}$ by the two pairs in this order: 
 $\bigl(A_i\oplus\langle x\rangle, B_i\bigr), \bigl(A_i, B_i\oplus\langle x\rangle\bigr)$,
so we obtain a new system $\mathcal{P}'$. Observe that $\mathcal{P}'$ is a skew Bollob\'as system.
Let $s_i=a_i+b_i$.  The contribution  of  
$(A_i, B_i)$ to $\omega(\mathcal{P})$ is  
$
\frac{1}{(s_i+1)\binom{s_i}{a_i}} .
$
For the new pairs $
\bigl(A_i\oplus\langle x\rangle, B_i\bigr)$ and $ \bigl(A_i, B_i\oplus\langle x\rangle\bigr)$, their contributions to  $\omega(\mathcal{P}')$  are 
$$
\frac{1}{(s_i+2)\binom{s_i+1}{a_i+1}} + \frac{1}{(s_i+2)\binom{s_i+1}{a_i}}=\frac{1}{(s_i+1)\binom{s_i}{a_i}}.
$$  
Thus, we have $\omega(P')=\omega(P)$.

\textbf{Step 3.  Iterating the extension until all pairs are full.} 
For a pair $(A_i, B_i)$, 
define its deficiency  $d_i = n - \dim(A_i+B_i)\geq 0$.  
The pair is  called full if $d_i=0$, i.e., $A_i\oplus B_i = V$.
Define the potential  function of $\mathcal{P}=\{(A_i, B_i)\}_{i=1}^m$ as
$$
\Phi(\mathcal{P}) = \sum_{i=1}^m 2^{\,n-d_i}.
$$
Assume that the current system $\mathcal{P}$ contains a non-full pair, say with deficiency $d_i \geq 1$.  
Performing the extension operation described in Step 2 replaces this pair by two new pairs, each having deficiency  $d_i-1$.  
The contribution of the original pair to $\Phi$ is  $2^{n-d_i}$. The two new pairs contribute $2 \cdot 2^{\,n-(d_i-1)} = 2^{n-d_i+2}$.  
Hence the change in the potential is  
$$
\Delta(\Phi) = 2^{n-d_i+2} - 2^{n-d_i}
          = 3\cdot 2^{n-d_i} > 0 .
$$
Since $m \leq 2^n$, we have  $\Phi(\mathcal{P})\leq m\cdot 2^n \leq 4^n$.  
Since $\Phi$ is a positive integer that strictly increases at  each extension operation, the process must terminate after finitely many steps.  
Termination occurs precisely when every pair is full, and the weight remains unchanged throughout.
Denote the resulting system again by $\mathcal{P}=\{(A_i, B_i)\}_{i=1}^m$.
Hence, we obtain a skew Bollob\'as system $\mathcal{P}=\{(A_i, B_i)\}_{i=1}^m$ in which  
$
A_i\oplus B_i = V\text{ for all }i\in [m],
$ 
and $\omega(\mathcal{P})=\max \omega(\tau)$.

At this stage, every pair satisfies
 $a_i+b_i = n$.
Fix a dimension $a$ and let $I_a = \{ i\in [m] : a_i = a\}$.  
Restricting   $\mathcal{P}$ to the index set $I_a$ yields a skew Bollob\'as system in which all $A_i$ have dimension $a$ and all $B_i$ have dimension  $n-a$.
By Theorem \ref{thm-L77}, we obtain  
$
|I_a| \le \binom{n}{a}.
$
It follows that
$$
\omega(\mathcal{P})= \sum_{i=1}^m \frac{1}{(n+1)\binom{n}{a_i}}
    = \frac{1}{n+1}\sum_{a} \sum_{i\in I_a} \frac{1}{\binom{n}{a}}\leq 1.\qedhere
$$
\end{proof}

\begin{lemma}\cite{F84}\label{Fle}
Let $n \geq k$ and $t = n - k$. 
Let $V$ be an $n$-dimensional real vector space. 
Let $W_1, \ldots, W_m$ be subspaces of $V$ with $\dim(W_i) < n$ for all $i$. 
Then there exists a $k$-dimensional subspace $V'$ such that  
$
\dim(W_i \cap V') = \max\{\dim(W_i) - t, 0\}
$  
for all $1 \leq i \leq m$.
\end{lemma}

The subspaces $V'$ guaranteed by Lemma \ref{Fle} is said to be \textit{in general position} with respect to the subspaces  $W_1, \ldots, W_m$.
Observe that if some $W_i$ satisfies $\dim(W_i) = n$, the statement remains true because in that case $\dim(V \cap V') =k= n-t$, and the formula still yields $\dim(W_i \cap V') = n-t$.

\begin{proof}[\bf Proof of Corollary \ref{cr1}]
Let  $\mathcal{P}=\{(A_i, B_i)\}_{i=1}^m$  be a skew Bollob\'as $t$-system of subspaces of $V$, where $\dim (A_i)=a_i$ and $\dim (B_i)=b_i$ for every $ i\in [m]$, and $a_m, b_1\geq t$. 
From the definition of a skew Bollob\'as $t$-system, we have $a_i>t$ and $b_j>t$ for all $i\in [m-1]$ and $j\in [2,m]$.  
It follows from Lemma \ref{Fle} that there exists an $(n-t)$-dimensional subspace $V'$ such that  
\begin{itemize}
\item[(a)] $\dim(A_i \cap V') = a_i - t \text{~and~} \dim(B_i \cap V') = b_i - t \text{ for all } i\in [m]$, 

\item[(b)] $\dim(A_i \cap B_i \cap V') = 0 \text{ for all } i\in [m],\text{~and~} \dim(A_i \cap B_j\cap V') >0 \text{ for all } 1\leq i<j\leq m$. 
\end{itemize}  
Hence, $\mathcal{P}'=\{(A_i\cap V', B_i\cap V')\}_{i=1}^m$ forms a skew Bollob\'as  system of subspaces of $V'$. 
Applying Theorem \ref{tm1} to $\mathcal{P}'$ yields the desired result.
\end{proof}

\subsection{Proof of Theorem \ref{tm2}}
To start with, we recall a uniform bound for skew Bollob\'as  system of $d$-tuples of subspaces, which will be used in the counting argument. This bound has its roots in Einstein \cite{Ein2008}, and was later modified by Oum and Wee \cite{OW2018} and recently rediscovered by Yue \cite{Y262}. 

\begin{theorem}[See \cite{Ein2008,OW2018,Y262}]\label{Y262} 
Let $\mathcal{P} = \{(A_i^{(1)},\dots,A_i^{(d)})\}_{i=1}^m$ be a skew Bollob\'as  system of $d$-tuples of subspaces of a vector space $V\cong \mathbb{R}^n$. Suppose that $\dim(A_i^{(k)})\le a_k $ for every $ i\in [m]$ and $k\in [d]$. Then
$$
m \leq \binom{a_1+\dots+a_d}{a_1,\dots,a_d}.
$$
\end{theorem}

The next lemma generalizes Lemma \ref{HFt1} to d-tuples of subspaces.

\begin{lemma}\label{le2} 
Let $\mathcal{P} = \{(A_i^{(1)},\dots,A_i^{(d)})\}_{i=1}^m$ be a skew Bollob\'as system of $d$-tuples of subspaces of  $V\cong \mathbb{R}^n$.
Then
$
m\leq d^{n}.$ 
\end{lemma}
\begin{proof}

Fix an index $i\in [m]$ and denote  
$
a_i^{(k)}=\dim(A_i^{(k)}),\ k\in [d].
$
Since $\dim\left(\sum_{k=1}^d A_i^{(k)}\right)=\sum_{k=1}^d \dim(A_i^{(k)})$, we have  
$
\sum_{k=1}^d a_i^{(k)} \leq n.
$  
Given a non-negative integer vector $(a_1,\dots,a_{d-1})$, define  
$$
\mathcal{I}(a_1,\dots,a_{d-1}):=\{\,i\in[m]: a_i^{(k)}=a_k, k\in [d-1]\}.
$$
Let $s=a_1+\dots+a_{d-1}$ and  $a_d=n-s\ge 0$.  
For every $i\in\mathcal{I}(a_1,\dots,a_{d-1})$, we have  
$
\dim(A_i^{(d)}) \le n-s = a_d
$
and $\dim(A_i^{(k)})= a_k$ for $k=1,\dots,d-1$ as well.  

Restricting   $\mathcal{P}$ to the index set $\mathcal{I}(a_1,\dots,a_{d-1})$ yields a  skew Bollob\'as system.  
By Theorem \ref{Y262}, we obtain  
$$
|\mathcal{I}(a_1,\dots,a_{d-1})|\le \binom{a_1+\dots+a_d}{a_1,\dots,a_d}
      =\binom{n}{a_1,\dots,a_d}.
$$
It follows that
\begin{align*}
m=&\sum_{\substack{a_1,\dots,a_{d-1}\ge 0\\ a_1+\dots+a_{d-1}\le n}}
     |\mathcal{I}(a_1,\dots,a_{d-1})|
    \leq \sum_{\substack{a_1,\dots,a_{d-1}\ge 0\\ a_1+\dots+a_{d-1}\le n}}
        \binom{n}{a_1,\dots,a_{d-1},n-\sum_{k=1}^{d-1}a_k}\\
=&\sum_{\substack{a_1,\dots,a_{d-1}\ge 0\\a_1+\dots+a_d=n}}\binom{n}{a_1,\dots,a_d}=d^n.\qedhere
\end{align*}
\end{proof}

The lemma below will allow us to define the extension operation for $d$-tuples of subspaces.

\begin{lemma}\label{le3} 
Let $\mathcal{P} = \{(A_i^{(1)},\dots,A_i^{(d)})\}_{i=1}^m$ be a skew Bollob\'as system of $d$-tuples of subspaces of  $V\cong \mathbb{R}^n$. Suppose that $(B^{(1)},\dots,B^{(d)})$ is a tuple of subspaces of $V$  such that for some  $i\in[m]$ and  some $k\in [d]$, 
 $B^{(k)}\neq A_i^{(k)}$, $B^{(\ell)}= A_i^{(\ell)}$ for all $\ell\in [d]\backslash \{k\}$  and  $\dim(B^{(1)}+\cdots +B^{(d)})=\dim(B^{(1)})+\cdots +\dim(B^{(d)})$. Then
$(B^{(1)},\dots,B^{(d)})\notin \mathcal{P}$.
\end{lemma}
\begin{proof}
For any $1\leq p\neq q\leq d$, we have 
$$
\dim(A_i^{(p)}\cap B^{(q)})=\begin{cases} \dim(B^{(p)}\cap B^{(q)})=0 & \text { if } q=k, \\ \dim(A_i^{(p)}\cap A_i^{(q)})=0 & \text { if } p=k,\\
 \dim(A_i^{(p)}\cap A_i^{(q)})=0 & \text { if } p\neq k \text{ and } q\neq k.
\end{cases}
$$
By the definition of a skew Bollob\'as system,  the result follows. 
\end{proof}

With these preparations we now prove Theorem \ref{tm2}.

\begin{proof}[\bf Proof of Theorem \ref{tm2}]
The argument follows the same three-step scheme that was used in the proof of Theorem \ref{tm1}.

\textbf{Step 1.  Existence of a  maximal-weight system.} 
For a skew Bollob\'as system $\mathcal{P} = \{(A_i^{(1)},\dots,A_i^{(d)})\}_{i=1}^m$
with $\dim(A_i^{(k)})=a_i^{(k)}$ and $s_i=a_i^{(1)}+\cdots+a_i^{(d)}$, define its  type
$$
\tau(\mathcal{P})=\bigl(m;(a_1^{(1)},\dots ,a_1^{(d)}),\dots ,(a_m^{(1)},\dots ,a_m^{(d)})\bigr)
$$
and its weight
$$
\omega(\mathcal{P})=
\sum_{i=1}^{m}
\frac{1}{\binom{s_i}{a_i^{(1)},\dots ,a_i^{(d)}}
      \binom{s_i+d-1}{d-1}}.
$$
The list $(a_1^{(1)},\dots ,a_1^{(d)}),\dots ,(a_m^{(1)},\dots ,a_m^{(d)})$ is ordered and repetitions are allowed.
By Lemma \ref{le2}, we have $m\le d^{n}$.
Let
$
D=\{(a_1,\dots ,a_d): a_i\in[0,n], \sum_{i=1}^{d}a_i\leq n\}.
$
Since $m\le d^{\,n}$ and each $(a_i^{(1)},\dots ,a_i^{(d)})\in D$, the number of different types is at most $\sum_{m=0}^{d^{\,n}}|D|^{m}$. Denote this finite set by $T$.
For a type $\tau=\bigl(m;(a_1^{(1)},\dots ,a_1^{(d)}),\dots ,(a_m^{(1)},\dots ,a_m^{(d)})\bigr)$, define
$$
\omega(\tau)=
\sum_{i=1}^{m}
\frac{1}{\binom{s_i}{a_i^{(1)},\dots ,a_i^{(d)}}
      \binom{s_i+d-1}{d-1}} .
$$
If $\tau$ is realizable (i.e., there exists a skew Bollob\'as system $\mathcal{P}$ with $\tau(\mathcal{P})=\tau$), then every system of that type has the same weight $\omega(\tau)$.
The set of realizable types is a subset of the finite set $T$, consequently,
$
\{\omega(\tau): \tau \text{ is realisable}\}
$
is a finite set of real numbers, and its maximum is attained.
Choose a realizable type $\tau^{*}$ with $\omega(\tau^{*})=\max\omega(\tau)$ and pick a system $\mathcal{P}$ whose type is $\tau^{*}$. Then
$
\omega(\mathcal{P})=\max\omega(\tau).
$

 \textbf{Step 2.  An extension operation that preserves weight.}  
Assume that in the chosen maximal‑weight system $\mathcal{P} = \{(A_i^{(1)},\dots,A_i^{(d)})\}_{i=1}^m$
 some tuple $(A_i^{(1)},\dots ,A_i^{(d)})$ is not full, i.e.,
$
\bigoplus_{k=1}^{d} A_i^{(k)}\neq V.
$
Take a vector $x\in V\setminus\bigoplus_{k=1}^{d} A_i^{(k)}$ and construct $d$ new tuples
$$
T_k=\bigl(A_i^{(1)},\ldots,A_i^{(k)}\oplus\langle x\rangle,\dots,A_i^{(d)}\bigr),\quad k\in [d].
$$
By Lemma \ref{le3}, we know that $T_k\notin \mathcal{P}$ for all $k\in [d]$.
Replacing  the original tuple $(A_i^{(1)},\dots ,A_i^{(d)})$ in $\mathcal{P}$ by  the $d$ tuples in this order: $T_1,\dots ,T_d$, the resulting system $\mathcal{P}'$ is still a skew Bollob\'as system.
Recall that $\dim(A_i^{(k)})=a_i^{(k)}$ and $s_i=a_i^{(1)}+\cdots+a_i^{(d)}$.
The contribution  of the  tuple $(A_i^{(1)},\dots ,A_i^{(d)})$
 to $\omega(\mathcal{P})$ is  
$$
\frac{1}{\binom{s_i}{a_i^{(1)},\dots ,a_i^{(d)}}
      \binom{s_i+d-1}{d-1}}.
$$
For  the  tuple $T_k$, the dimension vector becomes $(a_i^{(1)},\dots ,a_i^{(k)}+1,\dots, a_i^{(d)})$, its contribution to $\omega(\mathcal{P}')$ is
$$
\frac{1}{\binom{s_i+1}{a_i^{(1)},\dots ,a_i^{(k)}+1,\dots,a_i^{(d)}}\binom{s_i+d}{d-1}}.
$$
Because $\binom{s_i+1}{a_i^{(1)},\dots ,a_i^{(k)}+1,\dots,a_i^{(d)}}= \frac{s_i+1}{a_i^{(k)}+1}\binom{s_i}{a_i^{(1)},\dots ,a_i^{(d)}}$, the total contribution of the $d$ new tuples $T_1,\dots ,T_d$ to $\omega(\mathcal{P}')$ equals
$$
\sum_{k=1}^{d}\frac{a_i^{(k)}+1}{(s_i+1)\binom{s_i}{a_i^{(1)},\dots ,a_i^{(k)},\dots,a_i^{(d)}}\binom{s_i+d}{d-1}}
 =\frac{s_i+d}{(s_i+1)\binom{s_i}{a_i^{(1)},\dots ,a_i^{(d)}}\binom{s_i+d}{d-1}}=\frac{1}{\binom{s_i}{a_i^{(1)},\dots ,a_i^{(d)}}
      \binom{s_i+d-1}{d-1}}.
$$
Hence, $\omega(\mathcal{P}')=\omega(\mathcal{P})$.

\textbf{Step 3.  Iterating the extension until all tuples are full.}
Define the deficiency of the $i$-th tuple $(A_i^{(1)},\dots ,A_i^{(d)})$ by $d_i=n-(a_i^{(1)}+\cdots+a_i^{(d)})\ge0$. The tuple is called full if $d_i=0$ (i.e., $\bigoplus_{k=1}^{d}A_i^{(k)}=V$).
Define the potential function of  $\mathcal{P} = \{(A_i^{(1)},\dots,A_i^{(d)})\}_{i=1}^m$ as
$$
\Phi(\mathcal{P})=\sum_{i=1}^{m}2^{\,n-d_i}.
$$
Assume that in the chosen maximal‑weight system $\mathcal{P}$ some tuple $(A_i^{(1)},\dots ,A_i^{(d)})$ is not full, say with deficiency $d_i \geq 1$.
Performing the extension operation described in Step 2 replaces this tuple by $d$ new tuples, each having deficiency  $d_i-1$. 
The original tuple contributes $2^{n-d_i}$ to $\Phi$. The $d$ new  tuples together contribute $d\cdot2^{n-(d_i-1)}=d\cdot2^{n-d_i+1}$.
Thus, 
$$
\Delta\Phi = d\cdot2^{n-d_i+1}-2^{n-d_i}=2^{n-d_i}(2d-1)>0 .
$$
Since $m\le d^{\,n}$, we have $\Phi(\mathcal{P})\leq m\cdot2^{n}\leq d^{n}\cdot2^{n}$.
Because $\Phi$ is a positive integer that strictly increases with each extension operation, the process must terminate after finitely many steps.
Termination occurs exactly when every tuple is full.
All extension  operations preserve the weight and the skew Bollob\'as property. Therefore, the final system, which we still denote by $\mathcal{P} = \{(A_i^{(1)},\dots,A_i^{(d)})\}_{i=1}^m$, satisfies
$$
A_i^{(1)} \oplus \cdots  \oplus A_i^{(d)} = V\text{ for all } i\in [m],\quad
\omega(\mathcal{P})=\max\omega(\tau).
$$

At this point, $a_i^{(1)}+\cdots+a_i^{(d)}=n$ for all $i\in [m]$. Hence, the weight simplifies to
$$
\omega(\mathcal{P})=
\sum_{i=1}^{m}
\frac{1}{\binom{n}{a_i^{(1)},\dots ,a_i^{(d)}}\,\binom{n+d-1}{d-1}} .
$$
Given a non-negative integer vector $(a_1,\dots,a_{d-1})$, define  
$$
\mathcal{I}(a_1,\dots,a_{d-1})=\{\,i\in[m]: a_i^{(k)}=a_k, k\in [d-1]\}.
$$
For every $i\in I(a_1,\dots ,a_{d-1})$, we have $\dim(A_i^{(k)})= a_k$ for $k\in[d-1]$ and $\dim A_i^{(d)}=a_d:=n-(a_1+\cdots+a_{d-1})$.
Consider the subsystem  obtained by restricting the original system  $\mathcal{P}$ to the index set $\mathcal{I}(a_1,\dots,a_{d-1})$, which is still a skew Bollob\'as system.  
By Theorem \ref{Y262}, we obtain
$$
|I(a_1,\dots ,a_{d-1})|\le\binom{n}{a_1,\dots ,a_d}.
$$
Partitioning the sum according to the first $d-1$ dimensions gives
$$
\sum_{i=1}^{m}\frac{1}{\binom{n}{a_i^{(1)},\dots ,a_i^{(d)}}}
 =\sum_{(a_1,\dots ,a_{d-1})}\;
   \sum_{i\in I(a_1,\dots ,a_{d-1})}\frac{1}{\binom{n}{a_1,\dots ,a_d}} .
$$
Inside each fixed class,
$$
\sum_{i\in I(a_1,\dots ,a_{d-1})}\frac{1}{\binom{n}{a_1,\dots ,a_d}}
   \le\frac{|I(a_1,\dots ,a_{d-1})|}{\binom{n}{a_1,\dots ,a_d}}\le 1 .
$$
The number of classes equals the number of non‑negative integer solutions of $a_1+\dots +a_{d-1}\le n$, which is $\binom{n+d-1}{d-1}$.
Consequently,
$$
\omega(\mathcal{P})=
\sum_{i=1}^{m}
\frac{1}{\binom{n}{a_i^{(1)},\dots ,a_i^{(d)}}\,\binom{n+d-1}{d-1}}\leq \frac{1}{\binom{n+d-1}{d-1}}\sum_{(a_1,\dots ,a_{d-1})}1=1.
$$ 
Because the extension procedure does not change the weight, the original maximal‑weight system $\mathcal{P}$ satisfies $\omega(\mathcal{P})\le1$.
Hence, every skew Bollob\'as $d$-tuple system obeys the same inequality.
\end{proof}



 \section*{Acknowledgement}
 Lihua Feng was supported by the NSFC (Nos. 12271527 and 12471022). This work was also supported by NSF of Qinghai Province (No. 2025-ZJ-902T). Lu Lu was supported by the NSFC (No. 12371362).

\end{document}